\documentclass[lettersize,journal]{IEEEtran}

\usepackage[utf8]{inputenc} 
\usepackage[T1]{fontenc}    
\usepackage{url}            
\usepackage{amsfonts}       
\usepackage{amsmath} 
\usepackage{amssymb}  
\usepackage{amsthm}
\usepackage{mathrsfs}
\usepackage[ruled]{algorithm2e}
\usepackage{bm}
\usepackage{stfloats}
\usepackage{graphicx}
\usepackage{nicefrac}       
\usepackage{microtype}      
\usepackage[normalem]{ulem}

\newcommand{\mA}{\mathcal{A}}

\newcommand{\mC}{\mathcal{C}}

\newcommand{\mE}{\mathcal{E}}
\newcommand{\mQ}{\mathcal{Q}}
\newcommand{\mS}{\mathcal{S}}

\newcommand{\R}{\mathbb{R}}
\newcommand{\mbE}{\mathbb{E}}
\newcommand{\mbP}{\mathbb{P}}

\newcommand{\mO}{\mathcal{O}}
\newcommand{\tmO}{\tilde{\mathcal{O}}}
\newcommand{\tl}{\tilde{l}}
\newcommand{\tx}{\tilde{x}}

\newcommand{\tv}{\tilde{v}}
\newcommand{\dd}{\mathrm{d}}
\newcommand{\um}{\overline{m}}

\newcommand{\dX}{\mathrm{d}X}

\newcommand{\ds}{\mathrm{d}s}
\newcommand{\dt}{\mathrm{d}t}
\newcommand{\dtau}{\mathrm{d}\tau}
\newcommand{\dW}{\mathrm{d}W}

\newcommand{\tX}{\tilde{X}}

\newcommand{\real}{\mathbb{R}}

\newcommand{\amgf}[1]{\Phi(#1)}

\newcommand{\defeq}{:=}
\newcommand{\tr}[1]{\text{tr}(#1)}

\newcommand{\innerp}[1]{\left< #1 \right>}
\newcommand{\expect}[1]{\mathbb{E}\left( #1 \right)}
\newcommand{\expectw}[2]{\mathbb{E}_{#1}\left( #2 \right)}
\newcommand{\prob}[1]{\mathbb{P}\left( #1 \right)}

\renewcommand{\top}{\mathsf{T}}

\newtheorem{definition}{Definition}
\newtheorem{lemma}{Lemma}
\newtheorem{thm}{Theorem}

\newtheorem{assumption}{Assumption}

\newtheorem{problem}{Problem}

\begin{document}

\title{Concentration of Stochastic System Trajectories with Time-varying Contraction Conditions}

\author{Zishun Liu$^{1}$, Liqian Ma$^{1}$, Hongzhe Yu$^{2}$ and Yongxin Chen$^{1}$
\thanks{1: The authors are with Georgia Institute of Technology, Atlanta, GA 30332 
        {\tt\small \{zliu910\}\{mlq\}\{yongchen\}@gatech.edu}}%
\thanks{2: Hongzhe Yu is with PlusAI, Santa Clara, CA 95054 
        {\tt\small hyu419@gatech.edu}}
}

\maketitle

\begin{abstract}
We establish two concentration inequalities for nonlinear stochastic system under time-varying contraction conditions. The key to our approach is an energy function termed Averaged Moment Generating Function (AMGF). By combining it with incremental stability analysis, we develop a concentration inequality that bounds the deviation between the stochastic system state and its deterministic counterpart. As this inequality is restricted to single time instance, we further combine AMGF with martingale-based methods to derive a concentration inequality that bounds the fluctuation of the entire stochastic trajectory. Additionally, by synthesizing the two results, we significantly improve the trajectory-level concentration inequality for strongly contractive systems. Given the probability level $1-\delta$, the derived inequalities ensure an $\mO(\sqrt{\log(1/\delta))}$ bound on the deviation of stochastic trajectories, which is tight under our assumptions. Our results are exemplified through a case study on stochastic safe control.
\end{abstract}


\newcommand{\upsig}{\overline{\sigma}}
\newcommand{\upr}{\overline{r}}
\section{Introduction}

Safety is a fundamental requirement for various systems including autonomous vehicles, robots, power grids and more. In many scenarios, such systems are modeled as Itô stochastic systems to account for inherent stochastic disturbances. Within this domain, \textit{concentration analysis} characterizes high-probability bound on the deviation of stochastic trajectories from their deterministic counterparts \cite{chen2025concentration}. Since the deterministic trajectory behaviors are well-studied \cite{XC-SS:22,ames2019control}, concentration analysis has become an effective strategy for ensuring stochastic system safety \cite{akella2025risk}.

Unlike linear stochastic systems with Gaussian state distributions \cite{sarkka2019applied}, capturing the deviation distribution for nonlinear systems can be intractable. Consequently, concentration analysis on nonlinear systems typically leverages the evolution of specific distributional properties, where contraction theory serves as a rigorous theoretical foundation\cite{dani2014observer,tsukamoto2021contraction} 
Contraction theory was established for analyzing the behavior of system trajectories with respect to each other\cite{jouffroy2005some}. Within this framework, contraction metric is used to measure the distance between two trajectories, and contraction rate is to quantify the evolution of this distance. 
For nonlinear control systems such as manipulator, quadrotor and vehicles, state-feedback controllers such as TV-LQR and incrementally stabilizing controllers \cite{zamani2013controller,manchester2017control} can ensure the closed-loop system with desired time-varying contraction metics and rates over extensive regions. Accordingly, this paper considers the nonlinear stochastic system under time-varying contraction conditions. 

In the past, the concentration of stochastic contracting system trajectories is typically analyzed through incremental stability analysis (ISA) \cite{pham2009contraction,tsukamoto2020robust}, where a high-probability bound on the stochastic state deviation relies on its expectation bound derived from ISA. However, such approaches often yield overly conservative bounds when the probability level is high (e.g., >99.9\%) due to technical limitations. Moreover, this framework only only provides \textit{pointwise} bounds at single time instance, falling to characterize the concentration of the \textit{entire trajectory}, without resorting to time discretization \cite{wei2025conformal}. Recently, we fundamentally improved the tightness of the ISA-based bound in \cite{szy2024TAC}, and further developed trajectory-level concentration inequalities in \cite{liu2025safety}. However, these results remain restricted to systems with time-invariant contraction conditions. 

In this paper, we investigate the concentration behavior of stochastic system trajectories under time-varying contraction conditions, which generalizes the main results in \cite{szy2024TAC,liu2025safety}. Compared to traditional ISA, the core of our theoretical analysis is a novel function termed Averaged Moment Generating Function (AMGF) developed in \cite{szy2024TAC}. For the stochastic system state, we integrate the AMGF with path-length integral techniques to obtain tight probabilistic bounds on the deviation from its deterministic counterparts. For the concentration of the entire stochastic trajectory, we exploit AMGF properties within a martingale-based framework to tightly bound stochastic trajectory fluctuations. Furthermore, for strongly contracting systems, we synthesize both approaches to significantly sharpen the trajectory-level concentration inequality. Given the probability level $1-\delta$, the derived bounds exhibit only $\mO(\log(1/\delta))$ dependence, ensuring the effectiveness in safety-critical control systems. The proposed theoretical results are validated via a case study of stochastic safe control.

\textit{Notations}: We use $\mathbb{R}_{\geq0}$ to denote the set of non-negative real numbers, $\R^{n\times n}_{SPD}$ to denote the set of all the SPD matrices on $\R^{n\times n}$, and $\mS^{n-1}$ to denote the unit sphere: $\{x\in\R^n: \|x\|=1\}$. We use $\|\cdot\|$ to denote $\ell_2$ norm, $\|\cdot\|_M$ to denote the weighted $\ell_2$ norm with $M\in\R^{n\times n}_{SPD}$, and $\innerp{\cdot,\cdot}$ to denote the standard inner product. We use $\mbE$ to denote expectation, $\mbP$ to denote probability.  For $x\in\R$,  $\lceil x\rceil=\min_{a}\{a\in\mathbb{N}:a\geq x\}$.

\section{Problem Formulation and Preliminaries}
\label{sec: problem formulation} 
In this section, we begin with the system configuration, based on which we introduce the contraction theory and formulate the problems we seek to solve.

Consider the following continuous-time stochastic system
\begin{equation}
    \mathrm{d}X_t = f(X_t, u_t)\,\mathrm{d}t + g_t(X_t)\,\mathrm{d}W_t
    \label{eq: stochastic_system}
\end{equation}
where $X_t\in \real^n$ is the state at time $t$, $u_t \in \mathcal{U}\subset \real^p$ is an open-loop \footnote{We suppose that any state-feedback controller, if exists, has been incorporated into $f(X,\cdot)$.} bounded input at time $t$, $f(X_t, u_t, t)\,\mathrm{d}t$ is the drift term with $f: \real^n\times\real^p\times\real_{\geq0}\to\real^n$, $g_t(X_t)\,\mathrm{d}W_t$ is the diffusion term, and $W_t \in \real^m$ is a $m$-dimensional Wiener process. We impose standard Lipschitz continuity and linear growth conditions \cite[Theorem 5.2.1]{BO:13} to guarantee the existence of a solution to \eqref{eq: stochastic_system}. At the same time, we assume that the diffusion term $g_t$ is uniformly bounded.

\begin{assumption} \label{as: sigma}
    $\exists ~\sigma>0$ such that $g_t(X_t)g_t(X_t)^\top\preceq \sigma^2I$.
\end{assumption}

Intuitively, a stochastic trajectory $X_t$ of \eqref{eq: stochastic_system} exhibits stochastic fluctuation driven by random noises, but tends to \textit{concentrate} around its deterministic counterpart. To formalize this, consider the deterministic system
\begin{equation}\label{eq: deterministic_system}
  \textstyle  \dot{x}_{t} = f(x_t, u_t),
\end{equation}
which represents the noise-free realization of \eqref{eq: stochastic_system}. A deterministic trajectory $x_t$ of \eqref{eq: deterministic_system} and a stochastic trajectory $X_t$ of \eqref{eq: stochastic_system} are defined as \textit{associated} trajectories if they have the same initial state $x_0=X_0$ and the same input $u_t$. Under this convention, the concentration of $X_t$ can be measured by its deviation from the associated $x_t$. Contraction theory provides an effective analytical framework for investigating such deviations.

Given the deterministic system \eqref{eq: deterministic_system}, its contracting property is defined as follows.
\begin{definition} [Contracting System, \cite{tsukamoto2021contraction}] \label{def: contraction}
    The deterministic system \eqref{eq: deterministic_system} is said to be $c_t$-\textit{contracting} if for any trajectories $x_t$ of \eqref{eq: deterministic_system}, $\exists c_t\in\R$ and $M_t\in\R^{n\times n}_{SPD}$ such that:
        \begin{equation} \label{eq: contracting M}
    (\frac{\partial f(x_t,u_t)}{\partial x_t})^\top M_t + M_t\frac{\partial f(x_t,u_t)}{\partial x_t} + \dot{M_t} \preceq 2c_tM_t
\end{equation}
holds at time $t$, where $c_t$ is called as \textit{contraction rate} and $M_t$ is called as \textit{contraction metric}. Especially, the system is said to be strongly contracting if $c_t<0$ holds for any $t$.
\end{definition}

As discussed in the Introduction, the system dynamics $f$, integrated with appropriate state-feedback controllers, can have time-varying contraction properties. In this paper, we impose this assumption on the systems, which generalizes the stationary contraction assumptions in our previous work \cite{szy2024TAC,liu2025safety}.

\begin{assumption} \label{as: M}
    Given the terminal time $T$, there exist $c_t:[0,T]\to\R$ and $M_t:[0,T]\to\R^{n\times n}_{SPD}$ such that the system \eqref{eq: deterministic_system} is $c_t$-contracting with contraction metric $M_t$.
\end{assumption}

For associated trajectories $X_t$ of \eqref{eq: stochastic_system} and $x_t$ of \eqref{eq: deterministic_system} under Assumption \ref{as: M}, $\|X_t-x_t\|_{M_t}$ serves as a natural metric for the deviation between them. At any given time point $t$, the concentration behavior of the stochastic state $X_t$ can be quantified by the high-probability bound on $\|X_t-x_t\|_{M_t}$, which is formulated as the following problem:
\begin{problem} \label{problem: single bound}
    For associated trajectories $X_t$ of \eqref{eq: stochastic_system} and $x_t$ of \eqref{eq: deterministic_system} under Assumptions \ref{as: sigma}-\ref{as: M}, given a time point $t$ and a probability level $\delta\in(0,1)$, determine a tight bound $r_{\delta,t}: (0,1)\times \R_{\geq0}\to \R_{\geq0}$ such that $\prob{\|X_t-x_t\|_{M_t}\leq r_{\delta,t}}\geq 1-\delta$.
\end{problem}

The main challenge of Problem \ref{problem: single bound} is in the tightness of  $r_{\delta,t}$ under the time-varying contraction condition. Although a valid probabilistic bound for this problem can be derived via standard ISA\cite{pham2009contraction}, but it only has an $\mathcal{\mO}(\sqrt{1/\delta})$ dependence on $\delta$, which is overly conservative for safety-critical systems where $\delta$ is usually $\leq10^{-3}$. 

Moreover, Problem \ref{problem: single bound} is restricted to the \textit{single-time instance}, but not \textit{over the entire stochastic trajectory}. To fully capture the concentration behavior of the stochastic trajectory $X_t$, $t\in[0,T]$, it is essential to derive a \textit{tight tube} around its associated deterministic trajectory $x_t$ that probabilistically envelopes their deviation, as formulated below.
\begin{problem} \label{problem: traj-level}
    For associated trajectories $X_t$ of \eqref{eq: stochastic_system} and $x_t$ of \eqref{eq: deterministic_system} under Assumptions \ref{as: sigma}-\ref{as: M}, given a period $[0,T]$ and a probability level $\delta\in(0,1)$, determine a tight $\upr_{\delta,t}: [0,T]\to \R_{\geq0}$ such that $\prob{\|X_t-x_t\|_{M_t}\leq \upr_{\delta,t},~\forall t\leq T}\geq 1-\delta$.
\end{problem}

The difference between Problem \ref{problem: single bound} and \ref{problem: traj-level} is illustrated in Fig. \ref{fig: single t vs traj}. Although our previous work \cite{liu2025safety} derives a tight $\upr_{\delta,t}$ under stationary contraction conditions, it remains unclear whether it can be generalized to the scenarios under Assumption \ref{as: M}.

\begin{figure}[t]
 \centering
\includegraphics[width =0.49\linewidth]{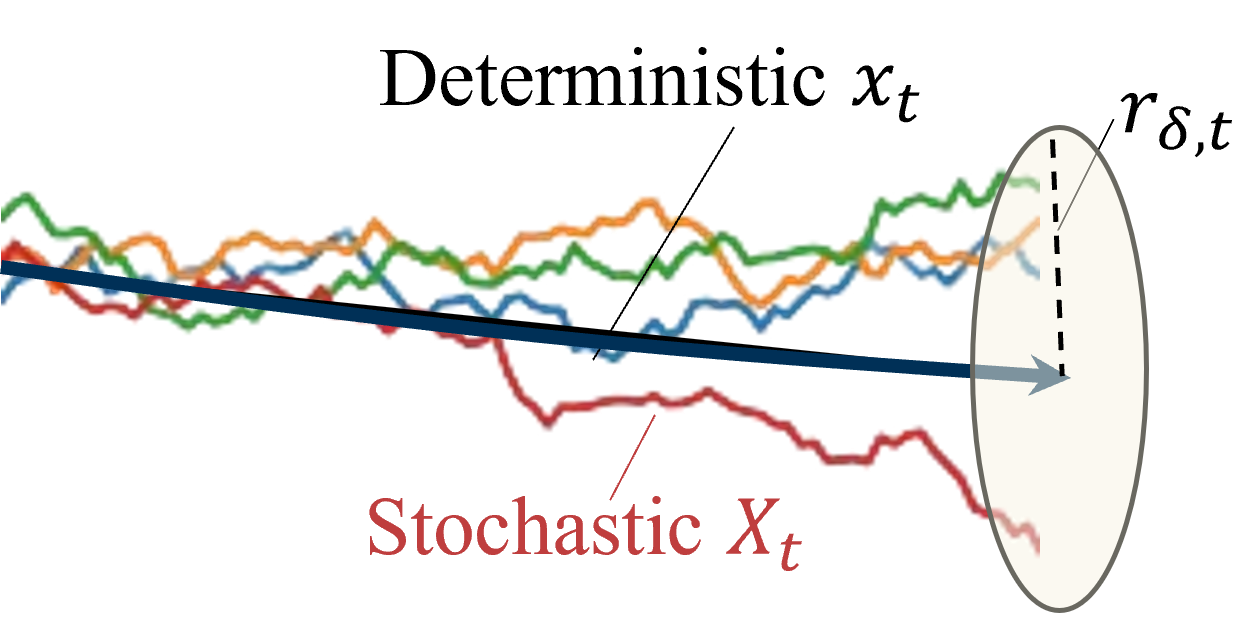}
\includegraphics[width =0.49\linewidth]{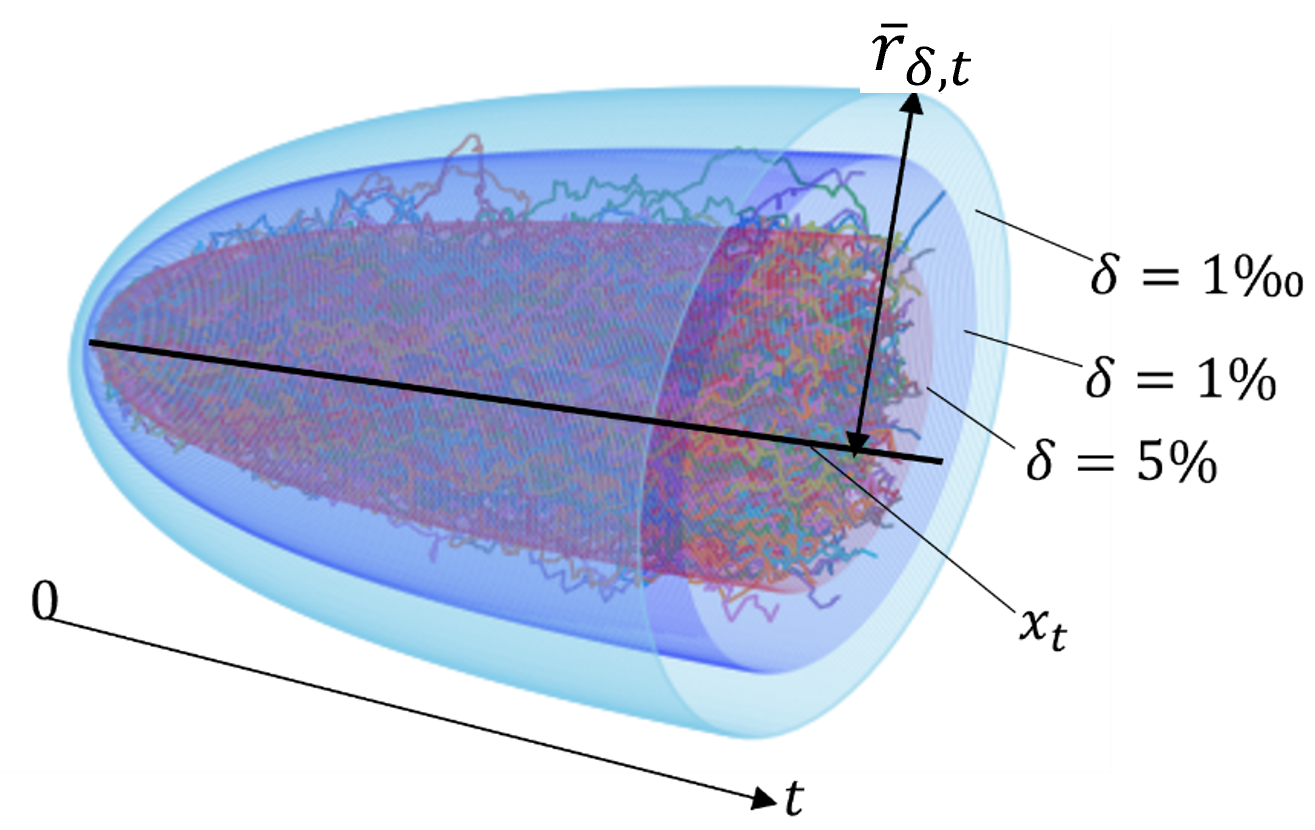}
  \caption{Comparison of the single-time bound $r_{\delta}$ (Left) and the trajectory-level bound $\upr_{\delta,t}$ (Right).
  }
	\label{fig: single t vs traj}
 \end{figure}

\section{Concentration of System State at Single Time}
In this section, we show the concentration of the stochastic system state $X_t$ at a single time $t$ by proposing a tight bound $r_{\delta,t}$ that solves Problem \ref{problem: single bound}. 

As noted in \cite{szy2024TAC}, the conservativeness of ISA stems from bounding the quadratic expectation $\mbE(\|X_t-x_t\|_{M_t}^2)$, which only ensures an $\mO(\sqrt{1/\delta})$ probabilistic bound for $\|X_t-x_t\|$ by Markov's Inequality. To circumvent this limitation, we leveraged a novel mathematical tool named \textit{Averaged Moment Generating Function} (AMGF) in \cite{szy2024TAC}, which improves the bound to $\mO(\sqrt{\log1/\delta})$ under stationary contraction metrics. The definition of AMGF and its weighted version are as follows:
\begin{definition} \label{def: AMGF}
    Given $X\in\R^n$, the averaged moment generating function (AMGF) is defined as 
    \begin{equation}
        \textstyle \mbE_X\Phi(X)=\mbE_X\mathbb{E}_{\ell\in\mS^{n-1}}e^{\lambda\innerp{\ell,X}},
    \end{equation}
    where $\Phi(X)=\mathbb{E}_{\ell\in\mS^{n-1}}e^{\lambda\innerp{\ell,X}}$ is defined as the energy function of AMGF. Moreover, given $M\in\R^{n\times n}_{SPD}$, we define the weighted version of $\Phi(X)$ as $\Phi_M(X)=\Phi(M^{1/2}X)$.
\end{definition}

Many intriguing properties of AMGF have been discovered in our previous works. Based on the properties demonstrated in \cite[Lemma 5.2]{szy2024TAC}. \cite[Lemma 5.4]{szy2024TAC} and \cite[Lemma 4.2]{liu2025safety}, it is straightforward to verify the following statements of $\Phi_M(X)$:
\begin{lemma} \label{lemma: M-AMGF}
    Consider the function $\Phi_M(X)$ in Definition \ref{def: AMGF}, where $X\in\R^n$ and $M\in\R^{n\times n}_{SPD}$, then it holds that: 
    \begin{enumerate}
    \item The value of $\Phi_M(X)$ merely depends on $\|X\|_M$.
    \item $\Phi_{M_1}(X_1)\leq\Phi_{M_2}(X_2)$ if $\|X_1\|_{M_1}\leq \|X_2\|_{M_2}$.
    \item  Given any $\varepsilon\in(0,1)$, $\Phi_M(X)\geq (1-\varepsilon^2)^{\frac{n}{2}}e^{\varepsilon \|\lambda X\|_M}.$
    \item If $X$ is random and $\exists \vartheta>0$ such that for any $\lambda\in\R$, $\mbE_X\left(\Phi_M\right)\leq e^{\frac{\lambda^2\vartheta^2}{2}}$,  
    then for any $\delta,\varepsilon\in(0,1)$:
    \begin{equation}
        \prob{\|x\|_M\leq \vartheta\sqrt{\varepsilon_1n +\varepsilon_2\log(1/\delta)}}\geq 1-\delta,
    \end{equation}
    \begin{equation} \label{eq: epsilon}
        \textstyle \text{where} ~ \varepsilon_1=\frac{\log(\frac{1}{1-\varepsilon^2})}{\varepsilon^2},~ \varepsilon_2=\frac{2}{\varepsilon^2}.
    \end{equation}
\end{enumerate}
\end{lemma}

Equipped with the weighted AMGF, we are ready to analyze the concentration of the stochastic system state $X_t$. Below we establish a tight probabilistic bound on the deviation to its associated deterministic trajectory $x_t$.

\begin{thm} \label{thm: single_bound}
    Consider a trajectory $X_t$ of the stochastic system 
     \eqref{eq: stochastic_system} 
    and its associated nominal trajectory $x_t$ of the system 
    \eqref{eq: deterministic_system}. Suppose that Assumption \ref{as: sigma} and \ref{as: M} hold for \eqref{eq: stochastic_system} and \eqref{eq: deterministic_system}. Define $\um_t$ as the maximal eigenvalue of $M_t$, $\upsig_t=\sqrt{\um_t}\sigma$, $\psi_t=\int_0^t c_\tau \dtau$ and $\Psi_t=\int_0^t \upsig_\tau^2e^{-2\psi_\tau}\dtau$. Then, for any $t\in[0,T]$, $\delta\in(0,1)$ and $\varepsilon\in(0,1)$:
    \begin{equation} \label{eq: single-time bound}
        \|X_t-x_t\|_{M_t}\leq \sqrt{e^{2\psi_t}\Psi_t(\varepsilon_1n+\varepsilon_2\log(1/\delta))}
    \end{equation}
    holds with probability $\geq1-\delta$, where $\varepsilon_1, \varepsilon_2$ are as \eqref{eq: epsilon}.
\end{thm}

\begin{proof}
 We adopt the path-length integral technique \cite{dani2014observer} to prove this theorem. Given the associated $x_t$ and $X_t$ at time $t$, and a smooth path $l_t(s):[0,1]\to\R^n$ connecting $x_t$ and $X_t$ (i.e., $l_t(0)=x_t$ and $l_t(1)=X_t$), define $v_t(s)=\frac{\partial l_t(s)}{\partial s}$ as the tangent vector at $l_t(s)$, and $V(s,t)=\Phi_{M_{t}}(v_t(s))$ as the energy density at $l_t(s)$. The total energy along the path $l_t$ is then defined as $\mathcal{E}(l_t)=\int_0^1V(s,t)\ds$. Finally, define the quadratic energy $\mathcal{Q}(l_t)=\int_0^1\|v_t(s)\|_{M_t}^2\ds$.

We start with the special case where the contraction rate $c_t\equiv0$. 
At time $t$, we consider the \textit{geodesic} $l^*_t(s)= \arg\min_{l_t} \mathcal{E}(l_t)$ among all the smooth paths connecting $x_t$ and $X_t$. After an infinitesimal time step $\dt$, consider the path $l_{t+\dt}(s)$ such that:
\begin{equation} \label{eq: l_t+dt}
    l_{t+\dt}(s)=l^*_t(s)+f(l^*_t(s),u_t)\dt+s\cdot g_t\dW_t,~\forall s\in[0,1],
\end{equation}
where $u_t$ and $g_t\dW_t$ are copies of that imposed on $X_{t}$. 
 Then, from \cite[Section 2.2.2]{tsukamoto2021contraction}, the variation from $v^*_t(s)$ to $v_{t+\dt}(s)$ can be modeled as: 
\begin{equation} \label{eq: dv_t}
    v_{t+\dt}(s) = v^*_t(s)+\frac{\partial f(l^*_t(s),u_t)}{\partial l^*_t(s)}v^*_t(s)+ g_t\dW_t.
\end{equation}

Based on \eqref{eq: dv_t}, we can analyze the variation of $\mathcal{E}(l^*_t)$. Define
\begin{equation*}
\mA(\mE(l^*_t))=\frac{\mbE(\mE(l_{t+\dt})|\mE(l^*_t))-\mE(l^*_t)}{\dt},
\end{equation*}
where the same operator $\mA(\cdot)$ is also applied to $V(s,t)$. Under this definition, we know $\mA(\mathcal{E}(l^*_t))=\int_0^1 \mA(V^*(s,t))\ds$.
To simplify $\mA(V^*(s,t))$, define $h^*(s,t)=\frac{\partial f(l^*_t(s),u_t)}{\partial l^*_t(s)}v^*_t(s)$. Then by Ito's Lemma, $\mA(V^*(s,t))$ can be unfolded as
\begin{equation*} \label{eq: AVst}
\begin{split}
\textstyle
    &\mA(V^*(s,t))= S_1+S_2,~ \text{where:} \\
    &S_1=\innerp{\frac{\partial V^*(s,t)}{\partial l^*_t(s)},f(l^*_t(s),u_t)}+\innerp{\frac{\partial V^*(s,t)}{\partial v^*_t(s)},h^*(s,t)}  \\
    &S_2=\frac{1}{2}\innerp{\nabla^2V^*(s,t),g_t(X_t)g_t(X_t)^\top}.
\end{split}
\end{equation*}
Define $l'_{t+\dt}(s)=l_{t+\dt}(s)|_{g_t=0}$, then by Ito's Lemma, $S_1=\frac{\mE(l'_{t+\dt})-\mE(l^*_t)}{\dt}$. Define $\dot{\mQ}(l^*_t)=\frac{\mQ(l'_{t+\dt})-\mQ(l^*_t)}{\dt}$ for the same noise-free case. By Lemma \ref{lemma: M-AMGF}, $V(s,t)$ is monotonely increasing with $\|v_t(s)\|_{M_t}$, so $l^*_t$ is also the geodesic with respect to $\mathcal{Q}(l_t)$, and $S_1$ and $\dot{\mQ}$ have the same sign. By \cite[Thm III.2]{singh2017robust}, we know $\dot{\mQ}(l^*_t)\leq 0$ when $c_t=0$. Therefore, $S_1\leq0$.

For the bound of $S_2$, define $\hat{v}_{s,t}=M_{t}^{1/2}v^*_t(s)$. Following the same steps as \cite[Thm 2 (24)]{liu2025safety}, we get
\begin{equation*}\label{eq: part 2 FPK}
    \begin{split}
        S_2=&\tfrac{1}{2}\innerp{\nabla^2V^*(s,t),g_t(X_t)g_t(X_t)^\top} \\
        =&\tfrac{1}{2} \expectw{\ell\sim\mS^{n-1}}{\innerp{\lambda^2e^{\lambda\innerp{\ell ,\hat{v}^*_t(s)}}M_{t}\ell\ell^{\top},g_t(X_t)g_t(X_t)^{\top}}} \\
        \leq& \tfrac{1}{2}\expectw{\ell\sim\mS^{n-1}}{\lambda^2e^{\lambda\innerp{\ell,\hat{v}^*_t(s)}}\tr{\ell\ell^{\top}}\,\|M_{t}g_t(X_t)g_t(X_t)^{\top}\|} \\
        \leq& \tfrac{\lambda^2\|M_t\|\sigma^2}{2}V^*(s,t) \leq \tfrac{\lambda^2\upsig_t^2}{2}V^*(s,t).
    \end{split}
\end{equation*}
Combining the bounds on $S_1$ and $S_2$, we get
\begin{equation} \label{eq: A(E(l*))}
\begin{split}
    \mA(\mathcal{E}(l^*_t)) \leq \tfrac{\lambda^2\upsig_t^2}{2}\int_0^1V^*(s,t)\ds \\
    = \tfrac{\lambda^2\upsig_t^2}{2}\int_0^1\mathcal{E}(l^*_t)\ds = \tfrac{\lambda^2\upsig_t^2}{2}\mathcal{E}(l^*_t),
\end{split}
\end{equation}
where the second row of \eqref{eq: A(E(l*))} follows the fact that for any $s\in[0,1]$, $\mathcal{E}(l^*_t)=V^*(s,t)$ due to the stationary-velocity property of the geodesic. Moreover, define 
$$\mA^*(\mathcal{E}(l^*_t))=\frac{\mbE(\mathcal{E}(l^*_{t+\dt})|\mathcal{E}(l^*_{t}))-\mathcal{E}(l^*_{t})}{\dt}.$$ 
Since $\mathcal{E}(l^*_{t+\dt})\leq \mathcal{E}(l_{t+\dt})$,  we know $\mA^*(\mathcal{E}(l^*_t))\leq \mA(\mathcal{E}(l^*_t))$. Combining the bound on $S_1$, $S_2$ and $\mA^*(\mathcal{E}(l^*_t))$, we get:
\begin{equation} \label{eq: E(l) martingale}
    \begin{split}
        \mA^*(\mathcal{E}(l^*_t))\leq \tfrac{\lambda^2\upsig_t^2}{2} \mathcal{E}(l^*_t),~ \mathcal{E}(l^*_0)=0,
    \end{split}
\end{equation}
which is a linear ODE that yields
\begin{equation}\label{eq: E<=}
        \mbE(\mathcal{E}(l^*_t))\leq e^{\frac{\lambda^2\int_0^t\sigma_\tau^2d\tau}{2}}.
    \end{equation}    
Notice that when $t$ is fixed, $M_t$ keeps stationary for different $v_t(s)$, which implies that $M_t$ defines a flat Riemannian metric at time $t$. Therefore, $l^*_t$ is a straight line \cite{singh2017robust}, $v^*_t(s)=X_t-x_t$ and $\mathcal{E}(l^*_t)=\Phi_{M_t}(X_t-x_t)$. Thus \eqref{eq: E<=} is equivalent to
\begin{equation}\label{eq: Phi|X-x|<=}
        \mbE(\Phi_{M_t}(X_t-x_t))\leq e^{\frac{\lambda^2\int_0^t\sigma_\tau^2d\tau}{2}}.
    \end{equation}    
From Lemma \ref{lemma: M-AMGF}, \eqref{eq: Phi|X-x|<=} implies that for $\forall \delta\in(0,1)$ and $\forall \varepsilon\in(0,1)$, with probability at least $1-\delta$, 
    \begin{equation}\label{eq: result c=0}
         \|X_t-x_t\|_{M_t}\leq \sqrt{(\varepsilon_1n+\varepsilon_2\log(1/\delta))\int_0^t\upsig_\tau^2d\tau},
     \end{equation}
    which corresponds to \eqref{eq: single-time bound} when $c_t\to0$. This completes the proof for the special case.

Next, we present the proof to the general case for any $c_t \in \mathbb{R}$ based on that for the special case $c_t = 0$.
 Define $\tX_t=e^{-\psi_t}X_t$ and $\tx_t=e^{-\psi_t}x_t$. 
Construct the geodesic $\tl^*_t(s)$ that connects $\tx_t$ and $\tx_t$, and define $\tl_{t+\dt}'$ such that $\tl_{t+\dt}'(s)=\tl^*_t(s)+f(l^*_t(s),u_t)\dt,~\forall s\in[0,1]$.
Following the same steps as \eqref{eq: l_t+dt}-\eqref{eq: dv_t}, we get
\begin{equation}
    \textstyle \frac{\tv'_{t+\dt}(s)-\tv^*_t(s)}{\dt}= -c_te^{-\psi_t}v^*_t(s)+e^{-\psi_t}\frac{\partial f(l^*_t(s),u_t)}{\partial l^*_t(s)}v^*_t(s).
\end{equation}
Therefore, the quadratic term $\tv^*_t(s)^\top M_t \tv^*_t(s)$ satisfies: 
\begin{equation*} \label{eq: scaling}
    \begin{split}
    \textstyle
        &\frac{\tv'_{t+\dt}(s)^\top M_{t+\dt} \tv'_{t+\dt}(s)-\tv^*_t(s)^\top M_t \tv^*_t(s)}{\dt}  \\
        =& \tv^*_t(s)^\top\dot{M_t}\tv^*_t(s) -2c_t\tv^*_t(s)^\top M_t\tv^*_t(s)^\top \\
        & +\tv^*_t(s)^\top \left( (\frac{\partial f(l^*_t(s),u_t)}{\partial l^*_t(s)})^\top M_t + M_t\frac{\partial f(l^*_t(s),u_t)}{\partial l^*_t(s)}\right)\tv^*_t(s) \\
    \end{split}
\end{equation*}
By Definition \ref{def: contraction}, we know
the equation above is $\leq0$, which implies that the system dynamics of $\tilde{x}_t$ is contracting with $\tilde{c}_t=0$ \cite{tsukamoto2021contraction}.
Then, notice that 
the diffusion term of the dynamics of $\tX_t$ satisfies $\|e^{-\psi_t}g_t(X_t)e^{-\psi_t}g_t(X_t)^{\top}\|\leq e^{-2\psi_t}\sigma^2\defeq \tilde{\sigma}_t^2$. Thus the derivation for the special case can be applied. Define $\tilde{l}^*_t(s)$ as the geodesic connecting $\tx_t$ and $\tX_t$. Following the same steps as the special case, we obtain that with probability at least $1-\delta$,
    \begin{equation}\label{eq: result tilde_c=0}
         \|\tX_t-\tx_t\|_{M_t}\leq \sqrt{(\varepsilon_1n+\varepsilon_2\log(1/\delta))\int_0^t \um_t\tilde\sigma_\tau^2d\tau}.
     \end{equation}
Recalling $X_t=e^{\psi_t}\tX_t$ and $x_t=e^{\psi_t}\tx_t$, we conclude the result of Theorem \ref{thm: single_bound}.
\end{proof}

We point out that $r_{\delta,t}$ derived from Theorem \ref{thm: single_bound} achieves an $\mO(\sqrt{\log 1/\delta})$ dependence on $\delta$ under time-varying contraction conditions, which is a fundamental improvement over the traditional ISA-based methods. 
Moreover, following \cite[Section V-E]{szy2024TAC}, it can be shown that our derived $r_{\delta,t}$ is the tightest obtainable bound under Assumptions \ref{as: sigma} and \ref{as: M}. Nevertheless, as aforementioned, the bound $r_{\delta,t}$ derived from Theorem \ref{thm: single_bound} only holds at single time instant, distinct from the trajectory-level bound $\upr_{\delta,t}$ required in Problem \ref{problem: traj-level}.

\section{Concentration of Stochastic Trajectories}
 For a safety-critical dynamical system, ensuring the safety of the entire trajectory is more significant than analyzing individual states.  
In this section, we show the concentration of the stochastic trajectory $X_t,~ t\in[0,1]$ by proposing a time-varying tube $r_{\delta,t}:[0,T]\to\R_{\geq0}$ that bounds the stochastic trajectory deviation, as required in Problem \ref{problem: single bound}.

\subsection{Martingale-Based Concentration Inequality}
For stochastic systems, martingale-based approaches are a standard paradigm for analyzing trajectory-level properties \cite{lavaei2022automated}. In particular, the \textit{affine martingale} has demonstrated its significance in addressing  concentration problems \cite{cosner2023robust,liu2025safety}. Its definition is as follows.

\begin{definition} [Affine Martingale, \cite{liu2025safety}] \label{def: CT AM}
   For a continuous stochastic process $Y_t,~ t\in[0,T]$, a nonnegative differentiable function $B(Y,t):\R^n\times[0,T]\to\R_{\geq0}$ is said to be an affine martingale (AM) of $Y_t$ if there exist $a_t\in\R,b_t\in\R_{\geq0}$ such that for all $t\leq T$ and the $\dt\to 0$: $\frac{\expect{B(Y_{t+\dt},t+\dt)|Y_t}-B(Y_t,t)}{\dt}\leq a_tB(Y_t,t)+b_t.$
\end{definition}

This type of semi-martingale was first introduced in \cite[Chapter 3]{1967stochastic}, and Definition \ref{def: CT AM} slightly generalizes it to time-varying coefficients $a_t,~b_t$. Based on the AM $B(Y_t)$, one can construct a sublevel set and quantify the probability of the trajectory $Y_t,~ t\in[0,T]$ staying in the set. This is formalized in the following lemma and its proof can be found in \cite[Lemma 4.1]{liu2025safety}.
\begin{lemma}\label{lemma: CT-AM}
    Consider a continuous-time stochastic trajectory $Y_t,~ t\in[0,T]$. Let $B(Y,t)$ be an AM of $Y_t$ with coefficients $a_t, b_t\geq0$.
    Define $\widetilde{B}(Y_t,t)=B(Y_t,t)\xi_t+\int_t^T b_\tau\xi_\tau \dd\tau$,
    where $\xi_t=e^{\int_t^Ta_\tau\dd\tau}$. Then given any $\overline{B}>0$ and the set $\mathcal{Y}_t=\{Y_t:~\widetilde{B}(Y_t,t)\leq \overline{B}\}$, it holds that $ \prob{Y_t\in\mathcal{Y}_t, \forall t\leq T}\geq 1-\frac{B(v_0,0)\xi_0+\int_0^Tb_\tau\psi_\tau\dd\tau}{\overline{B}}.$
\end{lemma}

If one can construct a proper AM for the deviation between associated trajectories $X_t$ of \eqref{eq: stochastic_system} and $x_t$ of \eqref{eq: deterministic_system}, then Lemma \ref{lemma: CT-AM} points to a solution to Problem \ref{problem: traj-level}. Leveraging both AM and AMGF, we now establish the following theorem, which characterizes the concentration of the stochastic trajectory.

\begin{thm} \label{thm: traj_bound}
    Consider a trajectory $X_t$ of the stochastic system 
     \eqref{eq: stochastic_system} 
    and its associated nominal trajectory $x_t$ of the system 
    \eqref{eq: deterministic_system}. Suppose that Assumption \ref{as: sigma} and \ref{as: M} hold for any trajectory of \eqref{eq: deterministic_system}. Given the terminal time $T$, $\delta\in(0,1)$ and $\varepsilon\in(0,1)$, define:
     \begin{equation}\label{eq: r traj}
        \textstyle \upr_{\delta,t}=
\sqrt{e^{2\psi_t}\Psi_T(\varepsilon_1n+\varepsilon_2\log(1/\delta))}
    \end{equation}
    where $\psi_t$, $\Psi_t$, $\varepsilon_1$, $\varepsilon_2$ are as Theorem \ref{thm: single_bound}. Then it holds that $\prob{\|X_t-x_t\|_{M_t}\leq \upr_{\delta,t}, \forall t\leq T}\geq 1-\delta.$
\end{thm} 
\begin{proof}
    The inequality \eqref{eq: E(l) martingale} implies that $\mathcal{E}(l^*_t)=\Phi_{M_t}(X_t-x_t)$ is an affine martingale defined in \cite{liu2025safety} with $a_t=\frac{\lambda^2\upsig_t^2}{2}$ and $b_t\equiv0$, and from \eqref{eq: Phi|X-x|<=} we know $\mathcal{E}(l^*_t)=\Phi_{M_t}(X_t-x_t)$. Therefore, for every $\upr>0$ and any $\eta\in\mS^{n-1}$, we get
\begin{equation}\label{eq: CT prob |v_t|<r_lam}
    \begin{split}
        &\prob{\|X_t-x_t\|_{M_t}\leq \upr, \forall t\leq T} \\
        =&  \prob{\|X_t-x_t\|_{M_t}\leq \|\upr\eta\|,~ \forall t\leq T}\\
        =&\prob{\Phi_{M_t}(X_t-x_t)\leq \amgf{\upr\eta}, \forall t\leq T} ~~~~ [\text{\textbf{Lemma \ref{lemma: M-AMGF}}}] \\
        =& \prob{e^{\frac{\lambda^2\int_t^T\upsig_\tau^2\dtau}{2}}\mathcal{E}(l^*_t)\leq e^{\frac{\lambda^2\int_t^T\upsig_\tau^2\dtau}{2}}\amgf{\upr\eta}, \forall t\leq T} \\
        \geq& \prob{e^{\frac{\lambda^2\int_t^T\upsig_\tau^2\dtau}{2}}\mathcal{E}(l^*_t)\leq \amgf{\upr\eta}, \forall t\leq T} \\
        \geq &1-\frac{e^{\frac{\lambda^2\int_0^T\upsig_\tau^2\dtau}{2}}}{\amgf{\upr\eta}},\quad \forall \eta\in\mS^{n-1} ~~~~[\textbf{Lemma \ref{lemma: CT-AM}}]\\
        \geq & 1- (1-\varepsilon^2)^{-\frac{n}{2}}\exp\left(\tfrac{\lambda^2\int_0^T\upsig_\tau^2\dtau}{2}-{\varepsilon\lambda \upr}\right) [\text{\textbf{Lemma \ref{lemma: M-AMGF}}}].
    \end{split}
\end{equation}
Minimizing the last line of \eqref{eq: CT prob |v_t|<r_lam} over $\lambda$, we get $\lambda^{*}=\frac{\varepsilon r}{\int_0^T\upsig_\tau^2\dtau}$. By Plugging $\lambda=\lambda^{*}$ into \eqref{eq: CT prob |v_t|<r_lam} and setting $$\upr=\sqrt{\tfrac{2\int_0^T\upsig_\tau^2\dtau}{\varepsilon^2}(\tfrac{n}{2}\log\left(\tfrac{1}{1-\varepsilon^2}\right)+\log(1/\delta))},$$
we arrive at the result of Theorem \ref{thm: traj_bound} in the case that $c\to0$. 

For general cases where $c_t\in\R$, we use the same scaling technique as that in the proof of Theorem \ref{thm: single_bound}. Define $\tX_t=e^{-\psi_t}X_t$ and $\tx_t=e^{-\psi_t}x_t$. From \eqref{eq: scaling} we know  the system dynamics of $\tilde{x}_t$ is contracting with $\tilde{c}=0$, and $\tX_t$ is perturbed by the stochastic with the diffusion term bounded by $\tilde{\sigma}_t^2=e^{-2\psi_t}\sigma^2$. Thus the derivation for the special case can be applied. Following the same steps as \eqref{eq: AVst}-\eqref{eq: E(l) martingale}, we know $\mE(\tl^*_t)=\Phi_{M_t}(\tX_t-\tx_t)$ is an affine martingale with $a_t=\frac{\lambda^2\um_t\tilde{\sigma}_t^2}{2}$ and $b_t\equiv0$. Then following \eqref{eq: CT prob |v_t|<r_lam} and the steps afterwards, we get
\begin{equation}\label{eq: CT r_scale}
    \prob{\|\tX_t-\tx_t\|_{M_t}\leq \tilde{r}_t, \forall t\leq T}\geq 1-\delta,
\end{equation}
where $\tilde{r}_t=\sqrt{\frac{\int_0^T \upsig_\tau^2e^{-2\psi_\tau}\dtau}{\varepsilon^2}(n\log\frac{1}{1-\varepsilon^2}+2\log(1/\delta))}$. Then we complete the proof by Recalling $X_t=e^{\psi_t}\tX_t$, $x_t=e^{\psi_t}\tx_t$.
\end{proof}

Similarly, the bound $\upr_{\delta,t}$ has $\mO(\sqrt{\log(1/\delta)})$ dependence on $\delta$, which scales slowly even when $\delta$ is very small. However, when the systems \eqref{eq: stochastic_system} and \eqref{eq: deterministic_system} are strongly contracting over a large horizon $T$, the term $\Phi_T$ in $\upr_{\delta,t}$ scales as $\tmO(e^{CT})$ with some positive constant $C$, thus dominating the coefficient $e^{2\psi_t}\Psi_T$ when $t\ll T$. This dominance renders the derived $\upr_{\delta,t}$ overly conservative.

\begin{figure}[t]
 \centering
\includegraphics[width =0.49\linewidth]{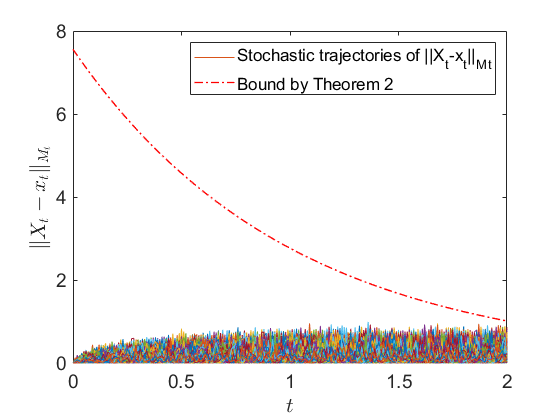}
\includegraphics[width =0.49\linewidth]{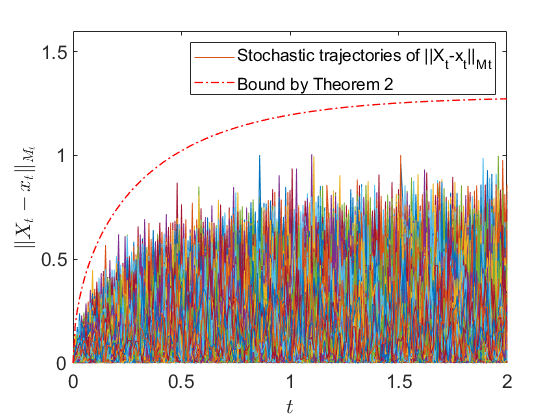}
  \caption{Comparison of the $\upr_{\delta,t}$ derived by Theorem \ref{thm: traj_bound} (Left) and Thorem \ref{thm: traj c<0}. The experiment is done on a linear system $\dX_t=A_tX_t+\Sigma\dW_t$ with strongly contracting $A_t$. Each figure contains 5000 independent trajectories of $\|X_t-x_t\|_{M_t}$, and use $\delta=0.001$, $\varepsilon=15/16$.
  }
	\label{fig: AM vs union}
 \end{figure}

\subsection{Improvement for Strongly Contracting Systems}
When the system is strongly contractive, the bound \eqref{eq: r traj} becomes conservative when $T$ is large. To overcome this drawback, we split $[0,T]$ into short intervals with length $\Delta t$, then apply Theorem \ref{thm: traj_bound} within each segment and Theorem \ref{thm: single_bound} at the end-points, and finally use union-bound inequality to obtain the $\upr_{\delta,t}$ for the entire trajectory. The improved result is stated in the following theorem.

\begin{thm} \label{thm: traj c<0}
    Consider a trajectory $X_t$ of the stochastic system 
     \eqref{eq: stochastic_system} 
    and its associated nominal trajectory $x_t$ of the system 
    \eqref{eq: deterministic_system}. Suppose that Assumption \ref{as: sigma} and \ref{as: M} hold with $c_t<0$. Given the terminal time $T$, $\delta\in(0,1)$, $\varepsilon\in(0,1)$ and $\Delta t>0$, define $k=\lceil \frac{t}{\Delta t}\rceil$, $\Psi_{t}^{\Delta t}=\int_{k\Delta t}^{(k+1)\Delta t} \upsig_\tau^2e^{-2\psi_\tau}\dtau$ and:
     \begin{equation}\label{eq: r traj c<0}
        \upr_{\delta,t}=
(\sqrt{e^{2\psi_t}\Psi_t}+\sqrt{\Psi_{t}^{\Delta t}})\sqrt{\varepsilon_1n+\varepsilon_2\log\frac{2T}{\delta \Delta t}}
    \end{equation}
     Then it holds that
    \begin{equation}
        \prob{\|X_t-x_t\|_{M_t}\leq \upr_{\delta,t}, \forall t\leq T}\geq 1-\delta.
    \end{equation}
\end{thm}

\begin{proof}
     Define $N=T/\Delta t$\footnote{We choose $\Delta t$ so that $N=T/\Delta t$ is an integer for convenience, but the same conclusion holds for arbitrary $\Delta t$.}. For $t=k\Delta t$, $k=1,\dots, N$, let 
 $$r_{k\Delta t}= \sqrt{e^{2\psi_{k\Delta t}}\Psi_{k\Delta t}(\varepsilon_1n+\varepsilon_2\log\frac{2N}{\delta})}.$$
 Then, by Theorem \ref{thm: single_bound}, for any $k$,
\begin{equation}\label{eq: t=kdt}
\begin{split}
    \mathbb{P}\left(\|X_{k\Delta t}-x_{k\Delta t}\|_{M_{k\Delta t}}\leq r_{k\Delta t}\right)\geq 1-\frac{\delta}{2N}. 
\end{split}
\end{equation}

For any $t\in(k\Delta t, (k+1)\Delta t)$, $k=0,\dots,N-1$, define a trajectory $y_t^{(k)}$ that satisfies
\begin{equation}
    \begin{split}
        \dot{y}_t^{(k)} = f(y_t^{(k)},d_t,t), \,\,\, y_{k\Delta t}^{(k)} = X_{k\Delta t}.
    \end{split}
\end{equation}
Then on the interval $(k\Delta t, (k+1)\Delta t)$, it holds that 
\begin{equation}\label{eq: sde xt'-xt}
    \dot{x}_t-\dot{y}_t^{(k)}=f(x_t,d_t,t)-f(y_t^{(k)},d_t,t).
\end{equation}
Since $f(x_t,d_t,t)$ is contracting, the contraction theory on Riemannian field \cite{singh2017robust} implies that
\begin{equation}\label{eq: xt'-xt}
\begin{split}
    &\|x_t-y_t^{(k)}\|_{M_t}\leq e^{\int_{k\Delta t}^t c_\tau\dtau}\|x_{k\Delta t}-y_{k\Delta t}^{(k)}\|_{M_{k\Delta t}} \\
    \leq &\|x_{k\Delta t}-y_{k\Delta t}^{(k)}\|_{M_{k\Delta t}}=\|x_{k\Delta t}-X_{k\Delta t}\|_{M_{k\Delta t}}.
\end{split}
\end{equation}

Note that $X_t$ and $y_t^{(k)}$ are associated trajectories over the time horizon $(k\Delta t, (k+1)\Delta t)$.
Let
$$\upr^\Delta=\sqrt{\Psi_{t}^{\Delta t}(\varepsilon_1n+\varepsilon_2\log\frac{2N}{\delta})}$$
and $\tilde{r}^{\Delta}=e^{\int_{k\Delta t}^t c_\tau\dtau}\upr^\Delta$, then $\tilde{r}^{\Delta}\leq\upr^\Delta$ and it holds that
\begin{equation} \label{eq: Xt'-xt}
    \begin{split}
        &\prob{\|X_t-y_t^{(k)}\|_{M_t}\leq \upr^\Delta,~ \forall t\in(k\Delta t, (k+1)\Delta t)} \\
        \geq &\prob{\|X_t-y_t^{(k)}\|_{M_t}\leq \tilde{r}^{\Delta},~ \forall t\in(k\Delta t, (k+1)\Delta t)} \\
        \geq &1-\frac{\delta}{2N},
    \end{split}
\end{equation}
 where the second ``$\geq$'' directly follows Theorem \ref{thm: traj_bound} by setting the time period as $\Delta t$ and the initial time as $k\Delta t$.

Next, by combining \eqref{eq: t=kdt} with \eqref{eq: Xt'-xt} and following the same union-bound technique as in \cite[Thm 3]{liu2025safety}, we can arrive the result of Theorem \ref{thm: traj c<0} to complete the proof.
\end{proof}

In comparison, the term $\Phi_T$ produced by Theorem \ref{thm: traj_bound} is replaced by the term $\Psi_{t}^{\Delta t}$ in Theorem \ref{thm: traj c<0}, which only scales as $\tmO(e^{C\Delta t})$ with some constant $C>0$,. Although the use of union-bound inequality causes an additional $\mO(\sqrt{\log\frac{T}{\Delta t}})$ factor, typically $e^{-cT}\gg e^{-c\Delta t}+\sqrt{\log\frac{T}{\Delta t}}$ for large $T$ and relatively small $\Delta t$. Consequently, Theorem~\ref{thm: traj c<0} provides a sharper bound on PT than Theorem~\ref{thm: traj_bound}. Figure \ref{fig: single t vs traj} illustrates this improvement via a simulation of the strongly contractive linear system $\dX_t=A_tX_t+\Sigma\dW_t$.

\section{Case Study}
In this case study, we demonstrate the application of the theoretical results in safety-aware stochastic system control, and exemplify it through a planar vertical take-off and landing (PVTOL) system with 99.99\% formal safety guarantee.

Consider the stochastic control system trajectory as
\begin{equation}\label{sys: control}
\begin{split}
    \dX_t&=f(X_t,K(X_t),u_t)\dt + g_t(X_t)\dW_t
    \\&:=f_{cl}(X_t,u_t)\dt+g_t(X_t)\dW_t,~~~ \text{given } X_0
\end{split}
\end{equation}
where $K(x)\in\R^p$ is a state-feedback controller, $f_{cl}(x,u)=f(x,K(x),u)$ is the closed-loop dynamics, and $u_t$ is a safety-aware openloop input. Given a safe set $\mC\subset\R^n$, a time period $[0,T]$ and a probability level $\delta$. By the set-erosion strategy \cite{liu2024safety}, to ensure that $\prob{X_t\in \mC, \forall t\leq T}\geq 1-\delta$, it is sufficient to design $K(x_t)$ and $u_t$ on the \textit{deterministic} $x_t$ such that
\begin{itemize}
    \item[$*$]   $\dot{x}_t=f_{cl}(x_t,u_t),~ x_0=X_0\in \mC$.
    \item[$*$]   $f_{cl}(x_t,u_t)$ satisfies Assumption \ref{as: M} with $c_t<0$.
    \item[$*$]  $x_t\in\mC\ominus \mathcal{B}(\upr_{\delta,t})$ $\forall t\leq T$, where $\upr_{\delta,t}$ is as Theorem \ref{thm: traj c<0}, $ \mathcal{B}(r)\in\R^n$ is a centered ball with radius $r$, and $\ominus$ denotes the Minkowskii difference. 
\end{itemize}

To exemplify this safety-aware control scheme, we consider a PVTOL system with additive stochastic disturbance modeled as \eqref{sys: control}. The formula of $f$ and the parameter settings are adopted from \cite{sun2021learning}. 
The planning task is a reach-avoid specification over horizon $T=2.5$ s.
The goal region is a circle in the \((x,z)\)-plane.
Obstacles consist of one box and two circular obstacles as visualized in Fig. \ref{fig: nominal traj PVTOL}.

To calculate the radius of the probabilistic tube, we consider a probability level $\delta = 10^{-4}$, $\varepsilon=0.9$ and $\Delta t=0.01$. We apply time-varying linear quadratic regulator (TVLQR) within the safe region to design the feedback control $K(x)$ and acquire the contraction metric $M_t$.
The time-varying contraction rate is estimated under this metric via sampling. This yields an ellipsoid probabilistic tube in the full state space: $\|X_t-x_t\|_{M_t} \leq \upr_{\delta,t}$, where $\upr_{\delta,t}$ is as described in \eqref{eq: r traj c<0}, and the nominal trajectory $\{x_t,u_t\},~t\leq T$ satisfies the control scheme ($*$).
Since the planning constraints (obstacles) are defined only in the $x$-$z$ position plane, we project the ellipsoid onto the $x$-$z$ plane and compute the radius of the outer approximating circle of the resulting 2D ellipse. In this example, the radius attains a maximum value of 0.54 over time, which is used to enlarge the obstacles and shrink the goal, as shown in Figure~\ref{fig: nominal traj PVTOL}. As a comparison, the bound $r_{\delta,t}$ derived by standard ISA \cite{pham2009contraction} is so large (>10) that it cannot be visualized in this experiment, indicating the tightness of our results.

\begin{figure}[t]
 \centering
 \includegraphics[width =0.51\linewidth]{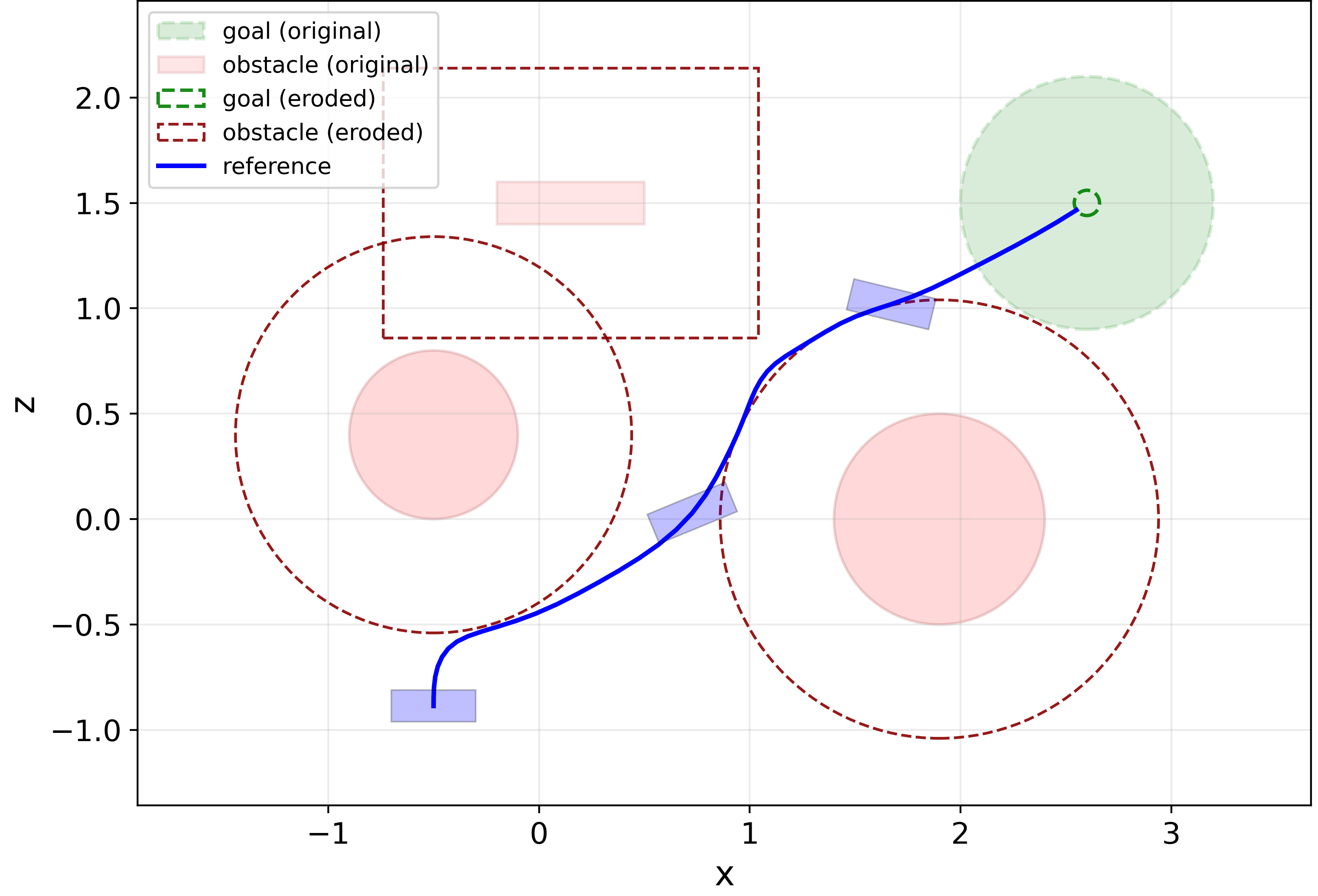}
 \includegraphics[width =0.47\linewidth]{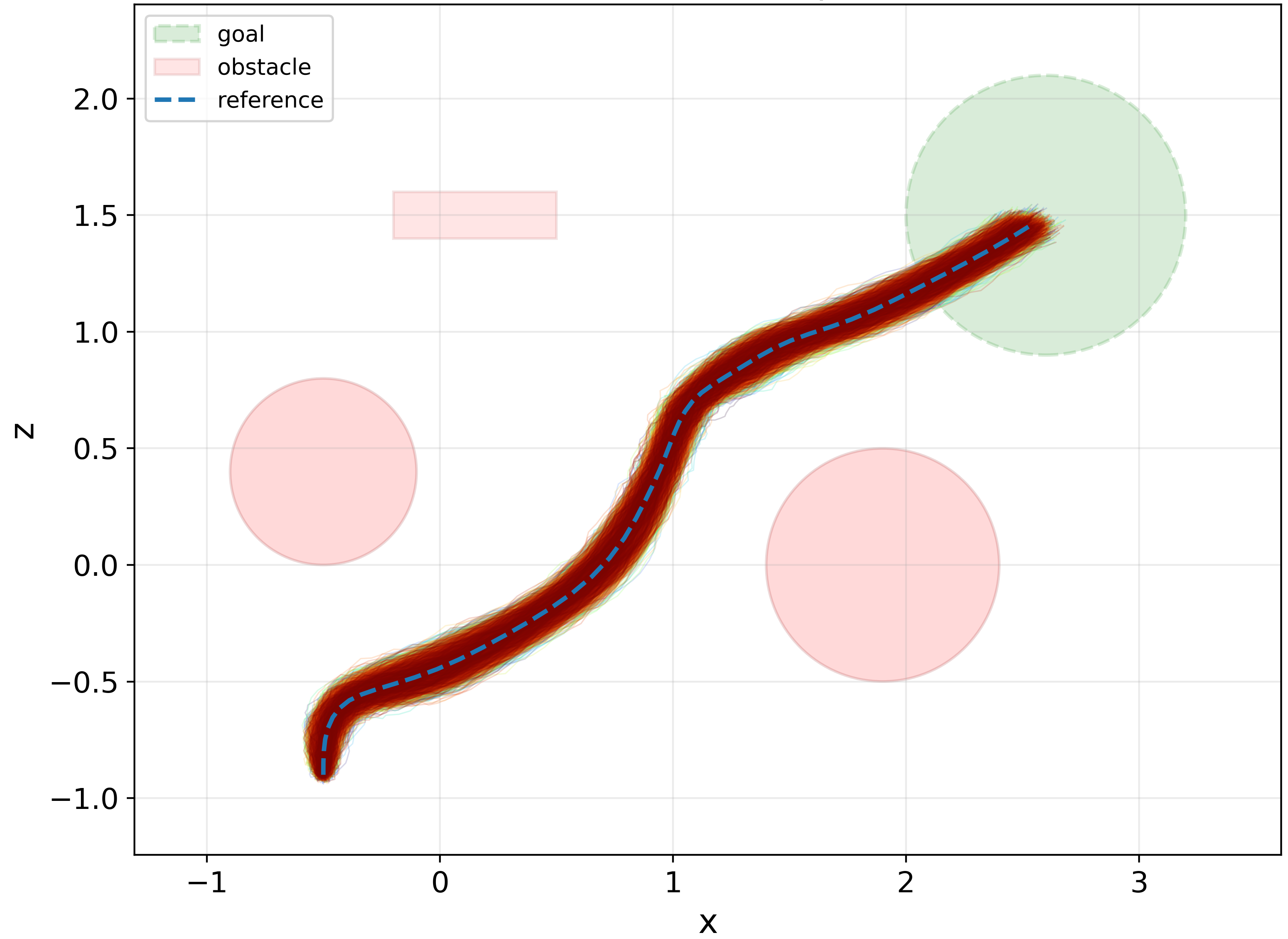}
  \caption{\textbf{Safe planning for the PVTOL system.} The robot must reach the goal region while avoiding obstacles. \textbf{Left}: The obstacles are enlarged according to the computed probabilistic bound \eqref{eq: r traj c<0}. The nominal trajectory of the PVTOL system remains collision-free and reaches the shrunken goal region. \textbf{Right}: $10^4$ stochastic rollouts of the TVLQR-controlled system.}
	\label{fig: nominal traj PVTOL}
\end{figure}

To validate the safe controller design scheme, we simulate $10^4$ stochastic trajectories. In Figure~\ref{fig: nominal traj PVTOL}, the nominal trajectory does not collide with the enlarged obstacles. As a result, all stochastic trajectories remain collision-free.

\section{Conclusion}
In this paper, we investigated the concentration behavior of stochastic system trajectories under time-varying contraction conditions. By combining a novel function termed AMGF with contraction theory and martingale-based methods, we established concentration inequalities for both stochastic system states and trajectories. Moreover, for strongly contracting systems, we significantly sharpened the trajectory-level bound by synthesizing Theorem \ref{thm: single_bound} and \ref{thm: traj_bound}. Our derived bounds achieve the tightest dependence on the probability level, which provides rigorous safety verification for safety-critical applications. The theoretical results were validated through an example of safety-aware stochastic control.

\bibliographystyle{ieeetr}
\bibliography{main}

@book{BO:13,
  title={Stochastic differential equations: an introduction with applications},
  author={{\O}ksendal, B.},
  year={2013},
  publisher={Springer Berlin, Heidelberg},
  series = {Universitext}, 
  doi = {10.1007/978-3-642-14394-6}, 
}

@article{liu2025safety,
      title={Safety Verification of Nonlinear Stochastic Systems via Probabilistic Tube}, 
      author={Zishun Liu and Saber Jafarpour and Yongxin Chen},
      journal={IEEE Transactions on Automatic Control},
      pages={1-15},
      year={2026}
}

@inproceedings{singh2017robust,
  title={Robust online motion planning via contraction theory and convex optimization},
  author={Singh, Sumeet and Majumdar, Anirudha and Slotine, Jean-Jacques and Pavone, Marco},
  booktitle={2017 IEEE International Conference on Robotics and Automation (ICRA)},
  pages={5883--5890},
  year={2017},
  organization={IEEE}
}

@article{szy2024TAC,
  title={Probabilistic Reachability Analysis of Stochastic Control Systems},
  author={S. Jafarpour$^{*}$ and Z. Liu$^{*}$ and Y. Chen},
  journal={IEEE Transactions on Automatic Control},
  volume={70},
  number={11},
  pages = {7080-7094},
  year={2025}
}

@article{tsukamoto2021contraction,
  title={Contraction theory for nonlinear stability analysis and learning-based control: A tutorial overview},
  author={Tsukamoto, Hiroyasu and Chung, Soon-Jo and Slotine, Jean-Jaques E},
  journal={Annual Reviews in Control},
  volume={52},
  pages={135--169},
  year={2021},
  publisher={Elsevier}
}

@article{manchester2017control,
  title={Control contraction metrics: Convex and intrinsic criteria for nonlinear feedback design},
  author={Manchester, Ian R and Slotine, Jean-Jacques E},
  journal={IEEE Transactions on Automatic Control},
  volume={62},
  number={6},
  pages={3046--3053},
  year={2017},
  publisher={IEEE}
}

@article{pham2009contraction,
  title={A contraction theory approach to stochastic incremental stability},
  author={Pham, Quang-Cuong and Tabareau, Nicolas and Slotine, Jean-Jacques},
  journal={IEEE Transactions on Automatic Control},
  volume={54},
  number={4},
  pages={816--820},
  year={2009},
  publisher={IEEE}
}

@book{1967stochastic,
  title={Stochastic Stability and Control},
  author = {Harold J. Kushner},
  year={1967},
  publisher={Academic Press}
}

@INPROCEEDINGS{cosner2023robust,
  title={Robust Safety under Stochastic Uncertainty with Discrete-Time Control Barrier Functions},
  booktitle={Robotics: Science and Systems (RSS)},
  author={Cosner, Ryan K and Culbertson, Preston and Taylor, Andrew J and Ames, Aaron D},
  year={2023}
}

@inproceedings{ames2019control,
  title={Control barrier functions: Theory and applications},
  author={Ames, Aaron D and Coogan, Samuel and Egerstedt, Magnus and Notomista, Gennaro and Sreenath, Koushil and Tabuada, Paulo},
  booktitle={2019 18th European control conference (ECC)},
  pages={3420--3431},
  year={2019},
  organization={IEEE}
}

@article{dani2014observer,
  title={Observer design for stochastic nonlinear systems via contraction-based incremental stability},
  author={Dani, Ashwin P and Chung, Soon-Jo and Hutchinson, Seth},
  journal={IEEE Transactions on Automatic Control},
  volume={60},
  number={3},
  pages={700--714},
  year={2014},
  publisher={IEEE}
}

@inproceedings{XC-SS:22,
  title={Reachability Analysis for Cyber-Physical Systems: Are We There Yet?},
  author={Chen, Xin and Sankaranarayanan, Sriram},
  booktitle={NASA Formal Methods: 14th International Symposium, NFM 2022, Pasadena, CA, USA, May 24--27, 2022, Proceedings},
  pages={109-130},
  year={2022},
  organization={Springer}
}

@article{chen2025concentration,
  title={Concentration of contractive stochastic approximation: Additive and multiplicative noise},
  author={Chen, Zaiwei and Maguluri, Siva Theja and Zubeldia, Martin},
  journal={The Annals of Applied Probability},
  volume={35},
  number={2},
  pages={1298--1352},
  year={2025},
  publisher={Institute of Mathematical Statistics}
}

@book{sarkka2019applied,
  title={Applied stochastic differential equations},
  author={S{\"a}rkk{\"a}, Simo and Solin, Arno},
  volume={10},
  year={2019},
  publisher={Cambridge University Press}
}

@inproceedings{jouffroy2005some,
  title={Some ancestors of contraction analysis},
  booktitle={Proceedings of the 44th IEEE Conference on Decision and Control},
  author={Jérôme Jouffroy},
  pages={5450--5455},
  year={2005},
  organization={IEEE}
}

@inproceedings{zamani2013controller,
  title={Controller synthesis for incremental stability: Application to symbolic controller synthesis},
  author={Zamani, Majid and van de Wouw, Nathan},
  booktitle={2013 European Control Conference (ECC)},
  pages={2198--2203},
  year={2013},
  organization={IEEE}
}

@article{lavaei2022automated,
  title={Automated verification and synthesis of stochastic hybrid systems: A survey},
  author={Lavaei, Abolfazl and Soudjani, Sadegh and Abate, Alessandro and Zamani, Majid},
  journal={Automatica},
  volume={146},
  pages={110617},
  year={2022},
  publisher={Elsevier}
}

@article{tsukamoto2020robust,
  title={Robust controller design for stochastic nonlinear systems via convex optimization},
  author={Tsukamoto, Hiroyasu and Chung, Soon-Jo},
  journal={IEEE Transactions on Automatic Control},
  volume={66},
  number={10},
  pages={4731--4746},
  year={2020},
  publisher={IEEE}
}

@article{liu2024safety,
  title={Safety verification of stochastic systems: A set-erosion approach},
  author={Liu, Zishun and Jafarpour, Saber and Chen, Yongxin},
  journal={IEEE Control Systems Letters},
  volume={8},
  pages={2859--2864},
  year={2024},
  publisher={IEEE}
}

@inproceedings{wei2025conformal,
  title={Conformal contraction for robust nonlinear control with distribution-free uncertainty quantification},
  author={Wei, Sihang and Ornik, Melkior and Tsukamoto, Hiroyasu},
  booktitle={2025 IEEE 64th Conference on Decision and Control (CDC)},
  pages={3502--3507},
  year={2025},
  organization={IEEE}
}

@inproceedings{sun2021learning,
  title={Learning certified control using contraction metric},
  author={Sun, Dawei and Jha, Susmit and Fan, Chuchu},
  booktitle={conference on Robot Learning},
  pages={1519--1539},
  year={2021},
  organization={PMLR}
}

@article{akella2025risk,
  title={Risk-aware robotics: Tail risk measures in planning, control, and verification [focus on education]},
  author={Akella, Prithvi and Dixit, Anushri and Ahmadi, Mohamadreza and Lindemann, Lars and Chapman, Margaret P and Pappas, George J and Ames, Aaron D and Burdick, Joel W},
  journal={IEEE Control Systems},
  volume={45},
  number={4},
  pages={46--78},
  year={2025},
  publisher={IEEE}
}
\end{document}